\documentclass[a4paper]{amsart}

\usepackage[latin1]{inputenc}
\usepackage[T1]{fontenc}
\usepackage{setspace}
\usepackage{amsmath, graphicx, cite, enumerate, comment}
\usepackage{amssymb, amsthm, amscd}
\usepackage{latexsym}
\usepackage[T1]{fontenc}	
\usepackage[english]{babel}
\usepackage{amsfonts}
\usepackage{amscd}
\usepackage{cite}

\newtheorem{theorem}{Theorem}[section]

\newtheorem{corollary}[theorem]{Corollary}

\newtheorem{lemma}[theorem]{Lemma}

\theoremstyle{definition}

\theoremstyle{remark}
\newtheorem{remark}[theorem]{Remark}

\numberwithin{equation}{section}

\begin{document}
\date{}
\title[Compactness of the space of minimal hypersurfaces]
{Compactness of the space of minimal hypersurfaces with bounded volume and $p-$th Jacobi eigenvalue}
\author{Lucas Ambrozio, Alessandro Carlotto and Ben Sharp}
\address{Department of Mathematics \\
	Imperial College of Science and Technology \\
	London}
\email{l.ambrozio@imperial.ac.uk, a.carlotto@imperial.ac.uk, benjamin.sharp@imperial.ac.uk}

\begin{abstract}
Given a closed Riemannian manifold of dimension less than eight, we prove a compactness result for the space of closed, embedded minimal hypersurfaces satisfying a volume bound and a uniform lower bound on the first eigenvalue of the stability operator. When the latter assumption is replaced by a uniform lower bound on the $p-$th Jacobi eigenvalue for $p\geq 2$ one gains strong convergence to a smooth limit submanifold away from at most $p-1$ points. 
\end{abstract}	

\maketitle

\maketitle

\section{Introduction} \label{sec:intro}

Let $(N^{n+1},g)$ be a closed Riemannian manifold and let us denote by $\mathfrak{M}^n(N)$ the class of closed, smooth and embedded minimal hypersurfaces\footnote{Throughout this paper, we shall always tacitly assume all hypersurfaces to be connected.} $M\subset N$. By the seminal work of Almgren-Pitts \cite{Pit81} (and Schoen-Simon \cite{SS81}) we know that such a set $\mathfrak{M}^n(N)$ is not empty whenever $2\leq n\leq 6$, the higher-dimensional counterpart of their method being obstructed by the occurrence of singularities of mass-minimizing currents. Over the last three decades, several existence results have been proven by means of equivariant constructions, desingularization, gluing and more recently, high-dimensional min-max techniques that ensure that in many cases of natural geometric interest the set $\mathfrak{M}^n(N)$ contains plenty of elements. Most remarkably, it was proven by Marques and Neves in \cite{MN13} that when $2\leq n\leq 6$ and the Ricci curvature of $g$ is positive, then $N$ contains at least countably many closed, embedded minimal hypersurfaces. Thus, one is led to investigate the \textsl{global structure} of the class $\mathfrak{M}^n(N)$ and the most basic question in this sense is perhaps that of finding geometrically natural and meaningful conditions that ensure the compactness of subsets of this space.
In the three-dimensional scenario, namely for $n=2$, and under an assumption on the positivity of the ambient Ricci curvature a prototypical statement was obtained, in 1985, by Choi and Schoen:

\begin{theorem}\cite{CS85}\label{thm:3dCS}
	Let $N$ be a compact 3-dimensional manifold with positive Ricci curvature. Then the space of compact embedded minimal surfaces of fixed topological type in $N$ is compact in the $C^{k}$ topology for any $k\geq 2$. Furthermore, if $N$ is real analytic, then this space is a compact finite-dimensional real analytic variety.
\end{theorem}	

Roughly speaking, the idea behind this result is that a uniform bound on the genus suffices for controlling both the area (as had already been observed in \cite{CW83}) and the second fundamental form of the minimal surfaces in question, in a uniform fashion.

When $n\geq 3$ new phenomena appear and a statement of this type cannot possibly be expected. Indeed, it was shown by Hsiang \cite{Hsi83} that in $S^{n+1}, n=3, 4, 5$ there exists a sequence of embedded minimal hyperspheres $M^{n}$ that have uniformly bounded volume and converge, in the sense of varifolds, to a singular minimal subvariety with two conical singularities located at antipodal points of the ambient manifold. Based on the seminal works of Schoen-Simon-Yau \cite{SSY75} and Schoen-Simon \cite{SS81} concerning stable minimal hypersurfaces, one is naturally led to conjecture that some sort of control on the spectrum of the Jacobi operator (together with a volume bound) should indeed suffice to obtain compactness.

A first result of this flavour, that holds true up to (and including) ambient dimension seven, was recently proven by the third-named author.

\begin{theorem}\cite{Sha15}\label{thm:sha}
Let $N^{n+1}$ be a smooth, closed Riemannian manifold with $Ric_{N}>0$ and $2\leq n\leq 6$. Then given any $0<\Lambda<\infty$ and $I\in\mathbb{N}$ the class
\[
\mathcal{I}(\Lambda, I):=\left\{M\in\mathfrak{M}^{n}(N):\ \mathcal{H}^{n}(M)\leq\Lambda, \  index(M)\leq I \right\}
\]
is compact in the $C^{k}$ topology for all $k\geq 2$. 
\end{theorem}

Here and above the word \textsl{compactness} is understood as single-sheeted graphical convergence to some limit $M\in\mathcal{I}(\Lambda, I)$.
As the reader can see, in analogy with Theorem \ref{thm:3dCS} here one does also need to assume positivity of the ambient Ricci curvature to derive a compactness theorem, for otherwise smooth, graphical convergence can only be ensured away from at most $I$ points (cmp. Theorem 2.3 in \cite{Sha15}).

In this article, we derive the rather surprising conclusion that \textsl{no assumption on the ambient manifold} is needed in proving a strong convergence theorem provided an upper bound on the Morse index is replaced by a lower bound on the first eigenvalue of the Jacobi operator. To state our result, we need to recall a definition: in the setting described above, we will say that $M_{k}\to M$ in the sense of smooth graphs at $p\in M$ if there exists $\rho>0, \eta>0$ such that \textsl{in normal coordinates centered at $p$} the intersection of $M_k$ with $B^{n}_{\rho}(0)\times B^{1}_{\eta}(0)$ consists, for $k$ large enough, of the collection of the graphs of smooth defining functions $u^1_k,\ldots, u^l_k$ with $u^j_{k}\to 0$ in $C^{m}$ for all $m\geq 2$ and $1\leq j\leq l$. We remark that if $M_k\to M$ in the sense of smooth graphs away from a finite set $\mathcal{Y}$ and $M$ is connected (and embedded) then the number of leaves of the convergence is constant.   

\begin{theorem}\label{thm:main}
	Let $2\leq n\leq 6$ and $N^{n+1}$ a smooth, closed Riemannian manifold. Denote by $\mathfrak{M}^n(N)$ the class of closed, smooth and embedded minimal hypersurfaces $M\subset N$. Let $\lambda_p(M)$ denote the $p$-th eigenvalue of the Jacobi operator for $M\in \mathfrak{M}^n(N)$. Given any $0<\Lambda <\infty$ and $0\leq  \mu <\infty $, define the class
	\[\mathcal{M}_p(\Lambda,\mu):=\{M\in \mathfrak{M}^n(N) : \text{$\mathcal{H}^{n}(M)\leq \Lambda$, $\lambda_p(M)\geq -\mu$}\}.
	\]
	Given a sequence $\left\{M_k\right\}\subset\mathcal{M}_{p}(\Lambda,\mu)$ there exists $M\in\mathcal{M}_{p}(\Lambda,\mu)$ such that $M_k\to M$ in the varifold sense and furthermore:
	\begin{enumerate}
	\item{if $p=1$ then $M_k\to M$ locally in the sense of smooth graphs;}
	\item{if $p\geq2$ then there exists a finite set $\mathcal{Y}=\left\{y_i\right\}_{i=1}^{P}$ with $P\leq p-1$ such that the convergence $M_k\to M$ is smooth and graphical for all $x\in M\setminus\mathcal{Y}$; if the number of leaves of the convergence is one then $\mathcal{Y}=\emptyset$.}	
	\end{enumerate}	
	
\end{theorem}

From such general assertion we can derive a strong compactness result under a purely topological assumption on the ambient manifold $N$.

\begin{corollary}
Let $2\leq n\leq 6$ and $N^{n+1}$ a smooth, closed Riemannian manifold not containing any one-sided minimal hypersurface (which holds true, for instance, if $N$ is simply connected). Then given any $0<\Lambda<\infty$ and $0\leq\mu<\infty$ the class $\mathcal{M}_1(\Lambda,\mu)$ is compact in the $C^{k}$ topology for all $k\geq 2$. 
\end{corollary}	

Indeed, in such scenario it is of course the case that the limit hypersurface $M$ (whose existence is ensured by Theorem \ref{thm:main}) is itself two-sided and thus for $k$ large enough $M_k$ is a finite covering thereof, hence the conclusion comes via an elementary topological argument due to the connectedness assumption on $M_k$.

In particular, we deduce from this statement that (in presence of a volume bound) the Morse index of an element in $\mathcal{M}_1(\Lambda,\mu)$ is uniformly bounded from above, which seems a rather unexpected conclusion from a purely analytic viewpoint, as we are considering an infinite family of elliptic operators parametrized by minimal hypersurfaces in $\mathcal{M}_1(\Lambda,\mu)$. On the other hand, we know that a bound on index does not give a bound on $\lambda_1$ (since even with bounded index we could have non-smooth convergence: to see this one only needs to consider a sequence of catenoids $M_{k}=\frac{1}{k}M$ in $\mathbb{R}^n$, all centred at the origin, and blowing down to a double plane. One gets that eventually the catenoid has index one in the unit ball (when it is scaled down sufficiently far, and only considering compactly supported variations), moreover that $\lambda_1 (M_k \cap B_1(0))\to -\infty$.) In particular, given a sequence with only volume and index bounds, we must have that $\lambda_1 \to -\infty$ in general - moreover the gaps between the eigenvalues must also diverge (since the index is bounded). 

Combining Theorem \ref{thm:main} with Remark \ref{rem:ricpos} we deduce the following:

\begin{corollary}
	Let $N^{n+1}$ be a smooth, closed Riemannian manifold with $Ric_{N}>0$ and $2\leq n\leq 6$. Then given any $0<\Lambda<\infty$ and $0\leq\mu<\infty$ and $p\geq 1$ the class $\mathcal{M}_p(\Lambda,\mu)$ is compact in the $C^{k}$ topology for all $k\geq 2$. 
\end{corollary}

From a different perspective, Theorem \ref{thm:main} gives us an interesting description of \textsl{what goes wrong} when absence of a smooth limit occurs for a family $\left\{M_k\right\}$ satisfying a uniform volume bound: necessarily, \textsl{every} eigenvalue of the Jacobi operator has to diverge to $-\infty$, which somehow captures the well-known picture of $M_{k}$ exhibiting some neck-pinching around finitely many points.

\begin{corollary}\label{cor:sing}
Let $\left\{M_k\right\}\subset\mathfrak{M}^{n}(N)$ be a sequence satisfying a uniform volume bound, so that possibly by extracting a subsequence we know \cite{Sim83} that $M_k\to M$ for some stationary, integral varifold $M$ in $N$. 
\begin{enumerate}
\item{If $M$ is not smooth, then $\lambda_p(M_k)\to-\infty$ as $k\to\infty$ for every $p\geq 1$.}
\item{If $M$ has multiplicity greater than one, then $\lambda_p(M_k)\to-\infty$ as $k\to\infty$ for every $p\geq 1$ provided $Ric_N>0$.}
\item{If we denote by $\mathcal{Y}$ the set of points of $M$ where the convergence $M_k\to M$ is not smooth and graphical, then $\lambda_p(M_k)\to-\infty$ for all $1\leq p\leq |\mathcal{Y}|$. In particular, if $|\mathcal{Y}|=\infty$ then $\lambda_p(M_k)\to-\infty$ for all $p\geq 1$.}	
\end{enumerate}
\end{corollary}	

For instance: this applies to the aforementioned Hsiang minimal hyperspheres \cite{Hsi83} (since they have a non-smooth limit) or more generally to the hyperspheres produced in \cite{Car15}. We further remark that the positivity assumption on the Ricci curvature of $N$ in item (2) above is essential because of the example described in Remark \ref{rem:mult2}. 

When $n=2$, the scenario we obtain by combining Theorem \ref{thm:3dCS} with Theorem \ref{thm:sha} and Theorem \ref{thm:main} is rather enlightening.

\begin{corollary}\label{cor:3dscenario}
	Let $\mathfrak{C}^{n}\subset\mathfrak{M}^{n}$ be a subclass of closed minimal hypersurfaces inside some smooth closed Riemannian manifold $N^{n+1}$ of dimension $2\leq n\leq 6$ satisfying $Ric_N >0$. Then a uniform bound on any one of the following quantities for every $M\in\mathfrak{C}^{n}$ leads to a bound on the rest of them for every $M\in\mathfrak{C}^{n}$:
	\begin{itemize}
		\item the genus of $M$ (when $n=2$)
		\item $index(M) + \mathcal{H}^n(M)$
		\item $\lambda_p(M) + \mathcal{H}^n(M)$
		\item $\sup_M |A| + \mathcal{H}^n(M)$
		\item $\int_M |A|^n+\mathcal{H}^n(M)$. 
	\end{itemize}
	\end{corollary}
	
Lastly, our main theorem extends (with minor variations in the proof) to the case when $N$ is a complete (not necessarily compact) Riemannian manifold, provided one replaces, in the statement, the set $\mathcal{M}_p(\Lambda,\mu)$ by the set
\[
\mathcal{M}_p^{\Omega}(\Lambda,\mu):=\{M\in \mathfrak{M}^n(N) : M\subset\overline{\Omega},  \ \text{$\mathcal{H}^{n}(M)\leq \Lambda$, \ $\lambda_p(M)\geq -\mu$}\}
\]
for some open, bounded domain $\Omega$.
In various situations of great geometric interest one can in fact drop the requirement that the minimal hypersurfaces in questions are contained in a given, bounded domain.

\begin{remark}
The conclusions of \ref{thm:main} also hold true when
\begin{enumerate}
\item{$(N^{n+1},g)$ is a compact Riemannian manifold with mean-convex boundary;}
\item{$(N^{n+1},g)$ is a complete Riemannian manifold such that for some compact set $K$ each component of $M\setminus K$ is foliated by closed, mean-convex leaves (in particular: asymptotically flat, asymptotically cylindrical and asymptotically hyperbolic manifolds).}	
\end{enumerate}	
The proof follows along the same lines of Theorem \ref{thm:main},  modulo exploiting a geometric maximum principle (see, for instance, \cite{SW89}) in order to reduce our compactness analysis to a bounded domain of the manifold $N$.
\end{remark}

\section{Preliminaries}\label{sec:prelim}

We shall recall here the definition of the Morse index and the Jacobi eigenvalues $\lambda_p$ for general smooth minimal  hypersurfaces $M\hookrightarrow N$. First of all, if $M$ is orientable then the second variation of the area functional can be written down purely in terms of section of the normal bundle $v\in \Gamma(NM)$ by
\[Q(v,v):=\int_{M} |\nabla^\bot v|^2 - |A|^2|v|^2 -Ric_N(v,v).
\]
Standard results on the spectra of compact self-adjoint operators on separable Hilbert spaces tell us that there is an orthonormal basis $\{v_i\}_{i=1}^{\infty}$ of $L^2(\Gamma(TN))$ consisting of eigenfunctions for the operator
\[
L^{\bot}v:=\Delta^{\bot} v +|A|^2 v + Ric^{\bot}_N(v)
\]
 with associated eigenvalues $\{\lambda_i\}_{i=1}^{\infty}$ of $Q$.
Moreover, we have the following \textsl{Rayleigh characterization of the eigenvalues} due to R. Courant:
\[ \lambda_{k} := \inf_{\textrm{dim}(V)=k}\max_{v\in V} \frac{Q(v,v)}{\int_M |v|^2}
\]
where of course $V$ is a linear subspace of $\Gamma(NM)$.  

Now, if $M$ is non-orientable then we simply lift the problem to its orientable double cover $\tilde{M}$ via $\pi :\tilde{M}\to M$. Consider the linear subspace of smooth sections $v\in\Gamma(\pi^{\ast}NM)$ such that $v\circ \tau =v$ where $\tau:\tilde{M}\to \tilde{M}$ is the unique deck transformation of $\pi$ which reverses orientation. Denote this subspace by $\tilde{\Gamma}(\pi^\ast NM)$. We can also pull back the quantities $|A(x)|^2:= |A(\pi(x))|^2$ and $Ric_N(a,b) := Ric_N (\pi_\ast a, \pi_\ast b)$. Thus consider the quadratic form
\[\tilde{Q}(v,v):=\int_{\tilde{M}} |\nabla^\bot v|^2 - |A|^2|v|^2 -Ric_N(v,v) 
\]
over $\tilde{\Gamma}(\pi^\ast NM)$. As before we can define the spectrum, and therefore index of $M$ to be that of $\tilde{M}$ with respect to $\tilde{Q}$ and $\tilde{\Gamma}(\pi^\ast NM)$.

\

If the ambient manifold $N$ is orientable, a closed hypersurface is orientable if and only if it is two-sided, namely if there exists a global section $\nu=\nu_M$ of its normal bundle inside  $TN$. If this is the case (namely if $M\hookrightarrow N$ is minimal and two-sided) the spectrum defined above patently coincides with the spectrum  of the \textsl{scalar} Jacobi or stability operator of $M$, namely 
\[
Lu:=\Delta_M u + (Ric_N(\nu,\nu)+|A|^2)u
\]
 when we regard $L:W^{1,2}(M)\to W^{-1,2}(M)$.
\

Furthermore, we shall introduce the following notation: given a minimal hypersurface $M$ in the Riemannian manifold $(N,g)$ and a bounded open domain $\Omega\subset M$ we shall set
\[
\lambda^{M}_1(\Omega)=\inf\left\{-\int_{M}v L^{\bot}v \ | v\in C_0^{\infty}(\Omega; NM) \ \textrm{and} \ \int_{M}|v|^2=1 \right\}
\]  
where $C_0^{\infty}(\Omega; NM) $ denotes the (smooth) sections of the normal bundle whose support, projected on the base $M$ is relatively compact in $\Omega$.

\section{Proofs}

For the sake of conceptual clarity, we will separate the proof of Theorem \ref{thm:main} in the case $p=1$ and $p\geq 2$, the former being a building block for the latter.

\begin{proof}[Proof of Theorem \ref{thm:main}, case $p=1$.]
	
Let $\left\{M_k\right\}\subset\mathcal{M}_1(\Lambda,\mu)$ be a sequence of closed, embedded minimal hypersurfaces satisfying our bounds on the volume and first Jacobi eigenvalue: we claim the existence of a constant $C=C(N,\Lambda,\mu)>0$ such that
\[
\sup_{k\geq 1} \sup_{z\in M_k}|A_k(z)|\leq C
\]
where $A_k$ denotes the second fundamental form of $M_k$ in $(N,g)$. For the sake of a contradiction, let us assume instead that such a uniform curvature bound does not hold. Then,  we could find (for every $k\geq 1$) a sequence of points $\left\{z_k\right\}\subset N$ such that $A_k$ attains its maximum value at $z_k\in M_k$ and, furthermore, $\lim_{k\to\infty}|A_k(z_k)|=+\infty$. Thanks to the compactness of the ambient manifold $N$, possibly by extracting a subsequence (which we shall not rename) we can assume that $z_k\to y$ for some point $y\in N$.
Let us pick, once and for all, a small radius $r_0>0$ (less than the injectivity radius of $(N,g)$ at $y$) and let us denote by $\left\{x\right\}$ a system of geodesic normal coordinates centered at $y$ and by $g_{ij}(x)$ the corresponding components of the Riemannian metric $g$. For $k$ large enough we know that $M_{k}\cap B_{r_0/2}(z_k)\subset B_{r_0}(y)$ and we can assume, without loss of generality, that $M_{k}\cap B_{r_0}(y)$ is two-sided. We can then consider the blown-up hypersurfaces defined by
\[
\hat{M}_{k}:= |A_k(z_k)|(M_k-z_k)
\]
and the appropriately rescaled Riemannian metrics on $\hat{B}_{k}:=B_{r_0 |A_k(z_k)|/2}(0)\subset\mathbb{R}^{n+1}$
\[
\hat{g}_k(x):=g\left(z_k+\frac{x}{|A_k(z_k)|} \right).
\]
(For the sake of clarity we have identified, in the equation above, the hypersurface $M_k$ with its portion in $B_{r_0/2}(z_k)$).
Now, the hypersurface $\hat{M}_{k}$ is minimal in metric $\hat{g}_k$ and patently satisfies volume and curvature bounds, for $M_k\in\mathcal{M}_1(\Lambda,\mu)$ implies
\[
\frac{\mathcal{H}^{n}(\hat{M}_k\cap B_r(0))}{r^{n}}\leq \Lambda, \ \textrm{for any} \ r< \frac{r_0 |A_k(z_k)|}{2}
\]
and by scaling
\[
\sup_{x\in\hat{B}_k}|\hat{A}_k(x)|\leq 1, \ \ \hat{A}_k(0)=1
\]
where $\hat{A}_k$ denotes the second fundamental form of $\hat{M}_k$ in metric $\hat{g}_k$. It follows that the sequence $\left\{\hat{M}_k\right\}$ converges (for any $m$ in $C^{m}_{loc}$, in the sense of smooth graphs) to a complete, embedded minimal hypersurface $\hat{M}_{\infty}\subset \mathbb{R}^{n+1}$ (in flat Euclidean metric). We further claim that $\hat{M}_{\infty}$ has to be stable. If not, we could find a smooth, compactly supported vector field $u$ such that the second variation 
\[
\left[\frac{d^2}{dt^2}\right]_{t=0}\mathcal{H}^{n}\left((\varphi_t)_{\#}\hat{M}_{\infty}\right)<0
\]
where $\left\{\varphi_t\right\}$ is the flow of diffeomorphisms generated by $u$ (which coincides with the identity outside of a compact set). As a result, thanks to the locally strong convergence $\hat{M}_k\to\hat{M}_{\infty}$ we would have
\[
\left[\frac{d^2}{dt^2}\right]_{t=0}\mathcal{H}^{n}\left((\varphi_t)_{\#}\hat{M}_{k}\right)<0
\]
for all indices $k$ that are large enough and, more specifically, for a fixed open, bounded set $\Omega\subset\mathbb{R}^{n+1}$ we would have $\lambda^{\hat{M}_{k}}_1(\Omega)\leq -\varepsilon$ for some $\varepsilon>0$. Therefore, scaling back and keeping in mind the Rayleigh characterization of the eigenvalues of an elliptic operator we must conclude that
\[
-\mu\leq \lambda^{M_k}_1(\Omega)\leq -\varepsilon |A_k(z_k)|,
\]
which is impossible when $k$ attains sufficiently large values. It follows that $\hat{M}_{\infty}$ is a \textsl{stable} minimal hypersurface in $\mathbb{R}^{n+1}$ (with polynomial volume growth), hence an affine hyperplane by the work of Schoen-Simon \cite{SS81} and thus on the one hand $\hat{A}_{\infty}$ vanishes identically, while on the other $\hat{A}_{\infty}(0)=1$ and this contradiction completes the proof of our initial claim. Once those uniform curvature estimates are gained, the strong convergence of $M_{k}\to M$ follows from a geometric counterpart of the Arzel\'a-Ascoli compactness theorem, and the fact that in this case the volume and first Jacobi eigenvalue of $M$ are also controlled, namely $M\in\mathcal{M}_1(\Lambda,\mu)$ is also clear.
\end{proof}

\begin{proof}[Proof of Theorem \ref{thm:main}, case $p\geq2$.]
	
We shall start by stating the following important:

\begin{lemma}\label{lem:parti}
Let $\Omega_1, \Omega_2,\ldots, \Omega_p\subset N$ be $p$ pairwise disjoint, bounded open sets. If we assume $M\in\mathcal{M}_p(\Lambda,\mu)$ and $M\cap \Omega_i\neq\emptyset$ for $i=1,\ldots,p$ then there exists an index $i_0$ such that $\lambda^{M}_1(\Omega_{i_0})\geq -\mu$.
\end{lemma}

\
Indeed, suppose that were not the case: then we could find, for each index $i=1,\ldots, p$ a section $\phi_{i}\in C^{\infty}_{0}(\Omega_i;NM)$ that ensures $\lambda^{M}_{1}(\Omega_i)<-\mu$ (that is to say $\int_M|\phi_i|^2=1$ and $Q(\phi_i,\phi_i)<-\mu$) and thus
\[
Q(\phi,\phi)<-\mu \ \textrm{for all} \ \phi\in W=\left<\phi_1,\ldots, \phi_p\right>_{\mathbb{R}} \ \textrm{such that} \ \int_M |\phi|^2 =1
\]
which contradicts the assumption that $\lambda_p(M)\geq -\mu$ due to the well-known min-max characterization of the eigenvalues of an elliptic operator.		
(In case $M$ is not orientable, one needs to consider the space $\tilde{W}=\left<\tilde{\phi_1},\ldots, \tilde{\phi_p}\right>_{\mathbb{R}}$ where each $\tilde{\phi_i}$ is the lift of $\phi_i$ to $\tilde{M}$ and the quadratic form $Q$ is evaluated on $\tilde{M}$ as explained in Section \ref{sec:prelim}).

\

Now, let a sequence $\left\{M_k\right\}\subset\mathcal{M}_p(\Lambda,\mu)$ be given: thanks to the volume bound, we know that possibly by taking a subsequence (which we shall not rename) $M_k\to \textbf{V}$ in the sense of varifolds, for some integral varifold $\textbf{V}$. Furthermore, given $\varepsilon>0$ the Lemma \ref{lem:parti} we have just stated and a standard covering argument ensure that there exists a set $\mathcal{Y}=\left\{y_i\right\}_{i=1}^{P}\subset\textrm{spt}(\textbf{V})$ consisting of at most $p-1$ points and a subsequence $\left\{M_{l(k)}\right\}$ such that in $N\setminus\mathcal{Y}$ the hypersurfaces $M_{l(k)}$ locally converge to $\textbf{V}$ strongly in the sense of smooth graphs. This descends from the fact that for any $z\in N\setminus\mathcal{Y}$ and $\varepsilon>0$ the sequence $M_{l(k)}$ satisfies a uniform bound on the \textsl{first} Jacobi eigenvalue, which in turn implies
\[
\sup_{k\geq 1}\sup_{x\in B_{\varepsilon}(z)}|A_{l(k)}(x)|\leq C=C(N,\Lambda,\mu)
\]
by following the argument that has been used to prove Theorem \ref{thm:main} in the case $p=1$. Here $A_{l(k)}$ stands for the second fundamental form of $M_{l(k)}$ in the ambient manifold $(N,g)$. In particular, this implies that the varifold $\textbf{V}$ is supported on a smooth submanifold $M$ away from finitely many points, namely those points belonging to the set $\mathcal{Y}$. 

 We further claim that for any $y\in\mathcal{Y}$ there exists $\varepsilon_0>0$ such that
\[
\lambda^{M}_1(B_{\varepsilon_0}(y)\setminus\left\{y \right\})\geq-\mu \ \ \ (*)
\]
If this claim were false, then we could find a smooth, normal vector field $u_0$, compactly supported in $B_{\varepsilon_0}(y)\setminus\left\{y \right\}$ and hence (say) supported in $B_{\varepsilon_0}(y)\setminus B_{\varepsilon_1}(y)$ for some $0<\varepsilon_1<\varepsilon_0$ such that
\[
\frac{Q(u_0,u_0)}{\int_{M}|u_0|^2}<-\mu.
\]
At that stage, we shall observe that it cannot be $\lambda^{M}_1(B_{\varepsilon_1}(y)\setminus\left\{y \right\})\geq-\mu$ either (for otherwise we would have gained property $(*)$ with $\varepsilon_1$ in lieu of $\varepsilon_0$) and hence, again, there is a smooth vector field $u_1$ that is compactly supported in $B_{\varepsilon_1}(y)\setminus\left\{y \right\}$ and hence (say) supported in $B_{\varepsilon_1}(y)\setminus B_{\varepsilon_2}(y)$ for some $0<\varepsilon_2<\varepsilon_1$ such that
\[
\frac{Q(u_1,u_1)}{\int_{M}|u_1|^2}<-\mu
\]
with the same notation as above. Of course, we can repeat this argument $p$ times, hereby getting sections $u_0,\ldots, u_{p-1}$ supported on smaller and smaller annuli, specifically $u_j$ shall be supported on $B_{\varepsilon_j}(y)\setminus B_{\varepsilon_{j+1}}(y)$ for $0<\varepsilon_p<\ldots<\varepsilon_0$. But we already know that $M_{l(k)}\to M$ on (the closure of) $B_{\varepsilon_j}(y)\setminus B_{\varepsilon_{j+1}}(y)$ for each $j\leq p-1$ and thus we derive (for $k$ large enough) the conclusion
\[
\lambda^{M_{l(k)}}_1(B_{\varepsilon_j}(y)\setminus B_{\varepsilon_{j+1}}(y))<-\mu, \textrm{for} \ \ j=0,1,\ldots,p-1 
\]
which contradicts our preliminary Lemma \ref{lem:parti}. This ensures the validity of $(*)$ for some suitable choice of $\varepsilon_0>0$. At that stage, let $\textbf{V}^{(i)}$ for each fixed $y_i\in\mathcal{Y}$ be a tangent cone to the varifold $\textbf{V}$ at $y_i$: necessarily $\textbf{V}^{(i)}$ has to be stable, for otherwise we could scale back and argue as in the proof of Theorem \ref{thm:main} to show that the bound $(*)$ cannot possibly hold. As a result, $\textbf{V}^{(i)}$ is a stable minimal hypercone in $\mathbb{R}^{n+1}$ and hence an hyperplane for $2\leq n\leq 6$ due to the classic work of J. Simons \cite{Sim68}. Therefore $\textbf{V}$ is regular (in fact, smooth) in a neighborhood of each point $y_i$ thanks to Allard's regularity theorem \cite{All72}, and we can conclude that its support $M$ is a smooth, minimal hypersurface in $(N,g)$, as we had to prove.

Varifold convergence directly implies that $\mathcal{H}^{n}(M)\leq\Lambda$, while the fact that $\lambda_p(M)\geq -\mu$ is more delicate as the set $\mathcal{Y}$ may not be empty. To that aim, we argue as follows. For small $r>0$ let $\eta^{(j)}_r$ be a smooth non-negative function on $M$ that vanishes on the geodesic ball of radius $r$  centered at $y_j$ and equals one outside of the ball of radius $2r$. For the sake of contradiction, suppose there exists $p$ linearly independent (and orthonormal) sections in $W^{1,2}(M;NM)$ (in fact $\tilde{W}^{1,2}(\tilde{M};\pi^{\ast}NM)$ when $M$ is not orientable), say $\phi_1,\ldots,\phi_p$ such that
\[
\frac{Q(\phi,\phi)}{\int_{M}|\phi|^2}<-\mu \ \textrm{on} \ V=\left<\phi_1,\ldots,\phi_p \right >_{\mathbb{R}}
\]
and set $\phi^r_i=\phi_i\prod_{j=1}^{P}\eta^{(j)}_r$. The fact that points have zero capacity in $\mathbb{R}^{n}$ for any $n\geq 2$ implies that on the subspace $V_r=\left<\phi^r_1,\ldots,\phi^r_p \right >_{\mathbb{R}}$ the inequality above must also hold for $r$ small enough, and hence (replacing each $\phi_i$ by $\phi^{r}_i$ and then applying the Gram-Schmidt process to the latter family, without further renaming) we can assume that the sections in question vanish on small geodesic neighborhoods of the points in the set $\mathcal{Y}$.
Now, such vector fields can be extended to a tubular neighborhood of $M\hookrightarrow N$ (without renaming) and since each $M_k$ has to be contained in that neighborhood for $k$ large enough we can define sections $\phi^{r,k}_i\in \Gamma(M_k;NM_k)$ by projecting those extended vector fields onto the normal bundle of $M_k\hookrightarrow N$.
 The strong convergence of $M_k$ to $M$ away from the points in $\mathcal{Y}$ together with the assumption $M_k\in\mathcal{M}_p(\Lambda,\mu)$ implies that
 we can find real coefficients $\alpha^k_1,\ldots, \alpha^k_{p}$ such that
\[
\sum_{i=1}^{p}\alpha^k_{i}\phi^{r,k}_i=0.
\]
Possibly by dividing the coefficients by $\max_i |\alpha^k_i|$ and renaming we can assume that $|\alpha^k_{i}|\leq 1$ for each index $i$ and $|\alpha^k_{i_0}|=1$ for some index $i_0$. Hence, squaring the previous equation and integrating over $M_k$ we get
\[
0=\int_{M_k}\left|\sum_{i=1}^{p}\alpha^k_{i}\phi^{r,k}_i \right|^2=\sum_{i,j}\alpha^k_i\alpha^k_j\int_{M_k} g\left(\phi^{r,k}_i,\phi^{r,k}_j\right)
\]
and by letting $k\to\infty$ the orhogonality of the family $\left\{\phi^r_1,\ldots, \phi^r_p\right\}$ implies that
\[
\sum_{i=1}^{p}(\alpha_i)^2=0
\]
where (possibly by extracting a subsequence, which we shall not rename) $\alpha^k_{i}\to\alpha_i$ as $k\to \infty$ for each $i=1,\ldots, p$. Thus $\alpha_i=0$ for each $i$ but on the other hand (by construction) $\sum_{i=1}^{p}(\alpha_i)^2\geq 1$, a contradiction. This proves that $M\in\mathcal{M}_p(\Lambda,\mu)$.
Lastly, the fact that single-sheeted convergence implies $\mathcal{Y}=\emptyset$ is a direct consequence of Allard's interior regularity theorem, \cite{All72}. Thereby, the proof is complete. 
\end{proof}

\begin{remark}\label{rem:ricpos}
When the number of leaves in the convergence of $M_k$ to $M$ is known, then one can deduce further information about the limit hypersurface $M$. Specifically:	
	\begin{itemize}
		\item{if the number of sheets in the convergence is one}
		\begin{itemize}
			\item if $M$ is two-sided and $M_k\cap M=\emptyset$ eventually then $M$ is stable 
			\item if $M$ is two-sided and $M_k\cap M\neq \emptyset$ eventually then $index(M)\geq 1$
		\end{itemize}
		\item {if the number of sheets in the convergence is at least two}
		\begin{itemize}
			\item if $N$ has $Ric_N >0$ then $M$ cannot be one-sided
			\item if $M$ is two-sided then $M$ is stable. 
		\end{itemize}
	\end{itemize}
All of these statements follow from variations on the same argument, which consists of conctructing a global section in the kernel of the Jacobi operator $L$ of $M$ (or a suitable lift, in the case of the third assertion) by appropriately renormalizing the distance function between $M$ and $M_k$ (first and second assertion) and two adjacent leaves of $M_k$ (third and fourth assertion). The reader can consult pages 10-13 of \cite{Sha15} for detailed arguments.	 
\end{remark}	

\begin{remark}\label{rem:mult2}
The assertion given in part (1) of Theorem \ref{thm:main} is \textsl{sharp} at that level of generality in the sense that one can provide explicit examples of Riemannian manifolds $(N,g)$ and sequences $\left\{M_k\right\}\subset\mathcal{M}_{1}(\Lambda,\mu)$ such that $M_k\to M$ in the sense of smooth graphs, but with multiple leaves. For instance: let $(N^3,g)$ be gotten by taking the quotient of the product manifold $(S^2\times\mathbb{R}, g_{round}\times dt^2)$ modulo the equivalence relation $(x,t)\sim (-x,-t)$. If we consider $\pi$ the associated Riemannian projection, then $\pi(S^{2}\times\left\{t\right\})$ is a totally geodesic, stable minimal sphere for any $t\neq 0$ while $\pi(S^{2}\times\left\{0\right\})$ is a stable minimal $\mathbb{R}\mathbb{P}^{2}$ and for any sequence $t_{k}\downarrow 0$ we have that $\pi(S^2\times\left\{t_k\right\})\to\pi(S^2\times\left\{0\right\})$ strongly in the sense of graphical, two-sheeted convergence. 	
	
\end{remark}	

\textsl{Acknowledgments}. The authors wish to express their gratitude to Prof. Andr\'e Neves for a number of enlightening conversations. During the preparation of this article, L. A. was supported by CNPq-Brazil, while A.C. and B. S. were supported by Prof. Andr\'e Neves through his Leverhulme and European Research Council Start Grant.


\begin{thebibliography}{99}
	
	\bibitem{All72}
	W.~K.~Allard, \textit{On the first variation of a varifold}, Ann. of Math. \textbf{95}(3) (1972), 417--491. 
	
	\bibitem{Car15}
	A.~Carlotto, \textit{Minimal hyperspheres of arbitrarily large Morse index}, Preprint (2015), arXiv:1504.02066
	
	\bibitem{CS85}
	H.~I.~Choi and R.~Schoen, \textit{The space of minimal embeddings of a surface into a three-dimensional manifold of positive Ricci curvature}, Invent. Math. \textbf{81} (1985), 387--394.
	
	\bibitem{CW83}
	H.~I.~Choi and A.~N.~Wang, \textit{A first eigenvalue estimate for minimal hypersurfaces}, J. Differential Geom. \textbf{18}(3) (1983), 559--562.
	
	\bibitem{Hsi83}
	W.~Y.~Hsiang, \textit{Minimal cones and the spherical Bernstein problem. I.}, Ann. of Math. \textbf{118}(1) (1983), 61--73. 
	
	
	\bibitem{MN13}
	F.~C.~Marques and A.~Neves, \textit{Existence of infinitely many minimal hypersurfaces in positive Ricci curvature}, Preprint arXiv:1311.6501, 2013.
	
	
	\bibitem{Pit81}
	J.~Pitts, \textit{Existence and regularity of minimal surfaces on {R}iemannian manifolds}, Princeton University Press, Princeton, N.J.; University of Tokyo Press, Tokyo. Mathematical Notes \textbf{27} (1981). 
	
	
	\bibitem{SS81}
	R.~Schoen and L.~Simon, \textit{Regularity of stable minimal hypersurfaces}, Comm. Pure Appl. Math. \textbf{34} (1981), 741--797.
	
	\bibitem{SSY75}
	R.~Schoen, L.~Simon and S.~T.~Yau, \textit{Curvature estimates for minimal hypersurfaces}, Acta. Math. \textbf{134}(1) (1975), 275--288. 
	
	\bibitem{Sim83}
	L.~Simon, \textit{Lectures on geometric measure theory}, Proceedings of the Centre for Mathematical Analysis, ANU \textbf{3} (1983). 
	
	\bibitem{Sha15}
	B.~Sharp, \textit{Compactness of minimal hypersurfaces with bounded index}, Preprint (2015) arXiv:1501.02703
	
	\bibitem{Sim68}
	J.~Simons, \textit{Minimal Varieties in Riemannian Manifolds}, Ann. of Math. \textbf{88} (1968), 62--105.
	
	\bibitem{SW89}
	B. ~Solomon, B. ~White,  \textit{A strong maximum principle for varifolds that are stationary with respect to even parametric elliptic functionals}, Indiana Univ. Math. J. 38 (1989), no. 3, 683-691. 
	
	
	
\end{thebibliography}
\end{document}